\documentclass[11pt, twoside, fullpage]{article}
\usepackage{amssymb, amsmath, amsthm, verbatim}
\usepackage[top=30mm, left=23mm, right=23mm, bottom=30mm]{geometry}
\usepackage{inputenc}
\usepackage{fancyhdr}
\usepackage{tikz}
\usepackage{verbatim}
\usetikzlibrary{positioning}
\usepackage{titling}
\usepackage[hyphens]{url}
\usepackage{hyperref}
\usepackage[hyphenbreaks]{breakurl}
\usepackage[T1]{fontenc}
\usepackage{currvita}
\usepackage{xcolor}
\predate{}
\postdate{}
\usepackage{lipsum}
\newtheorem{theorem}{Theorem}

\newtheorem{lemma}[theorem]{Lemma}
\newtheorem{proposition}[theorem]{Proposition}

\newcommand{\ex}{\mathrm{ex}}
\newcommand{\exc}{\mathrm{exc}}
\begin{document}
\newcommand{\Addresses}{{
\bigskip
\footnotesize
\medskip

\noindent David~Ellis, \textsc{School of Mathematics, University of Bristol, The Fry Building, Woodland Road, Bristol, BS8 1UG, United Kingdom.}\par\noindent\nopagebreak\textit{Email address: }\texttt{david.ellis@bristol.ac.uk}

\medskip 

\noindent Maria-Romina~Ivan, \textsc{Department of Pure Mathematics and Mathematical Statistics, Centre for Mathematical Sciences, Wilberforce Road, Cambridge, CB3 0WB, United Kingdom.}\par\noindent\nopagebreak\textit{Email address: }\texttt{mri25@dpmms.cam.ac.uk}

\medskip

\noindent Imre~Leader, \textsc{Department of Pure Mathematics and Mathematical Statistics, Centre for Mathematical Sciences, Wilberforce Road, Cambridge, CB3 0WB, United Kingdom.}\par\noindent\nopagebreak\textit{Email address: }\texttt{i.leader@dpmms.cam.ac.uk}}}

\pagestyle{fancy}
\fancyhf{}
\fancyhead [LE, RO] {\thepage}
\fancyhead [CE] {DAVID ELLIS, MARIA-ROMINA IVAN AND IMRE LEADER}
\fancyhead [CO] {TUR\'AN DENSITIES FOR SMALL HYPERCUBES}
\renewcommand{\headrulewidth}{0pt}
\renewcommand{\l}{\rule{6em}{1pt}\ }
\title{\Large{\textbf{TUR\'AN DENSITIES FOR SMALL HYPERCUBES}}}
\author{DAVID ELLIS, MARIA-ROMINA IVAN AND IMRE LEADER}
\date{}
\maketitle
\begin{abstract}
How small can a set of vertices in the $n$-dimensional hypercube $Q_n$ be if it meets every
copy of $Q_d$? The asymptotic density of such a set (for $d$ fixed and $n$ large) is denoted by
$\gamma_d$. It is easy to see that $\gamma_d \leq 1/(d+1)$, and it is known that $\gamma_d=1/(d+1)$ for $d \leq 2$, but it was recently shown that $\gamma_d < 1/(d+1)$ for
$d \geq 8$. In this paper we show that the latter phenomenon also holds for $d=7$ and $d=6$.
\end{abstract}
\section{Introduction}
As usual, we let $Q_n$ denote the graph of the $n$-dimensional discrete cube, i.e., the graph whose vertex-set is $\{0,1\}^n$, and where two 0-1 vectors are joined by an edge if they differ in just one coordinate. For $d,n \in \mathbb{N}$, we define 
$$\exc(n,Q_d)= \max\{|\mathcal F|:\ \mathcal F \subset \{0,1\}^n,\ \mathcal F \text{ is }Q_d\text{-free}\}.$$ In other words, $\exc(n,Q_d)$ is the maximum possible size of a set of vertices in $Q_n$ that does not contain the vertex-set of a $d$-dimensional subcube. We define the \textit{vertex-Tur\'an density of $Q_d$} to be
$$\lambda(Q_d) = \lim_{n \to \infty}\frac{\exc(n,Q_d)}{2^n}.$$ A standard averaging argument shows that the above sequence of quotients is monotone decreasing, so the limit does indeed exist.
\par It is in fact more convenient to consider the (equivalent) `complementary' problem: what is the minimal size $g(n,d)$ of a subset of $\{0,1\}^n$ that intersects the vertex-set of every $d$-dimensional subcube? Writing
$\gamma_d = \lim_{n \to \infty} g(n,d)/2^n$, we have $\gamma_d = 1-\lambda(Q_d)$.\\
\par Note that we certainly have $\gamma_d \leq 1/(d+1)$, as identifying $\{0,1\}^n$ with $\mathcal{P}([n])$ in the usual way, we may take the collection of all sets of
size 0 mod $d+1$. It was conjectured  (see \cite{blm}) that $\gamma_d = 1/(d+1)$ for every integer $d \geq 1$, but this was disproved in \cite{eil} for all $d \geq 8$. On the other hand, it was proved by Johnson and
Entringer \cite{je} and also by Kostochka \cite{k} that $\gamma_2 = 1/3$, and it is trivial that $\gamma_1=1/2$. So it is of great interest
to understand the cases $3 \leq d \leq 7$.
\par Our aim in this paper is to resolve the question for $d=7$ and $d=6$. We show that $\gamma_d$ is strictly less than $1/(d+1)$ for these values. Our proofs build on the linear algebra approach in \cite{eil}. For $d=7$, the argument is similar to the arguments in \cite{eil}, with one crucial extension. But for $d=6$ the argument is considerably more involved. One of its main ingredients is an adaptation of an idea in Grebennikov and Marciano \cite{gm}, which we will describe when we come to the proof below.
\par It turns out that the notion of a `daisy' is key to understanding the hypercube Tur\'an problem. As usual, if $X$ is a set and $r \in \mathbb{N}$, we write $X^{(r)}$ for the set of all $r$-element subsets of $X$. For $s,t \in \mathbb{N}$ with $t \geq s$, we define an \textit{$(r,s,t)$-daisy} to be an $r$-uniform hypergraph of the form $\{S \cup X:\ X \in T^{(s)}\}$, where $S$ is an $(r-s)$-element set and $T$ is a $t$-element set disjoint from $S$. We denote this hypergraph by $\mathcal D_r(s,t)$, and call $S$ the \textit{stem} of the daisy. (We sometimes also call $T$ the `petal-set', as the sets $X$ in the definition are naturally viewed as petals.) For background on daisies see \cite{blm} and \cite{eil}, although we stress that this paper is essentially self-contained.
\par It would be extremely interesting to know what happens for $d=3$, $4$ and $5$. The case of $d=3$ is of especial interest, because that is a case where it is known that direct daisy constructions (as in Section 2 below, or in \cite{eil}) cannot apply -- see \cite{eil} for details. We do not know if any of the ideas in Section 3 below can help with this question.

\section{The Tur\'an density of $Q_7$}

\par We start with the case of $Q_7$. Here the key new idea is as follows: the construction below avoids
not only an `ordinary' $\mathcal D_r(2,t)$ daisy but also the daisy that would correspond to the `opposite part of a hypercube', namely a $\mathcal D_r(t-2,t)$ daisy. This is what allows us to use a more economical family in $Q_n$ than was used in \cite{eil}.

First, some standard terminology. In what follows, we will often identify $\{0,1\}^n$ with the power-set $\mathcal{P}([n])$ in the usual way (that is, by identifying a set $S \subset [n]$ with its indicator vector). The discrete cube $\{0,1\}^n \cong \mathcal{P}([n])$ splits up naturally into layers, the $r$th layer being (identified with) $[n]^{(r)}$. We will sometimes refer to this layer as the layer at `level' $r$.

\begin{proposition} Let $t$ be an integer such that $t-2$ is a prime power, say $q$. Recall that $\prod_{k=1}^{\infty}(1-q^{-k})$ is a positive constant. Then for any positive integer $r$ and any integer $n\geq r$, there exists a family $\mathcal F_n^r$ of $r$-element subsets of $[n]$ such that $\mathcal F_n^r$ is $\mathcal D_r(2,t)$-free, $\mathcal F_n^r$ is $\mathcal D_r(t-2,t)$-free, and
$$\frac{|\mathcal F_n^r|}{{n \choose r}}\geq\prod_{k=1}^{\infty} (1-q^{-k}).$$
\label{dual}
\end{proposition}
\begin{proof} Let $\ex(n,r,t)$ be the size of the largest family that is both $\mathcal D_r(2,t)$-free and $\mathcal D_r(t-2,t)$-free. Let $\mathcal G$ be a family of $r$-element subsets of $[n]$ that satisfies the above two conditions and is of maximal size, i.e., $|\mathcal{G}| = \ex(n,r,t)$. Note then that $\mathcal G_i=\{A\in\mathcal G:i\notin A\}$ is a family that is free of the two forbidden daisies, with ground set of size $n-1$, so $|\mathcal G_i|\leq \ex(n-1,r,t)$. Summing over all $i\in[n]$ we get $(n-r)|\mathcal G|\leq n \, \ex(n-1,r,t)$, showing that the sequence $\ex(n,r,t)/\binom{n}{r}$ is decreasing in $n$. Therefore, it is enough to find families with the required properties for arbitrarily large $n$.
\par To start with, let $n=q^r-1$ and identify $[n]$ with $\mathbb{F}_q^r \setminus \{0\}$, where $\mathbb{F}_q$ is the field of order $q$. Take $\mathcal F$ to consist of all bases of $\mathbb{F}_q^r$. It is easy to see that
\begin{equation}\label{eq:bases-estimate}\frac{|\mathcal{F}|}{{n \choose r}} > \prod_{k=1}^{r}(1-q^{-k});\end{equation}
for a proof of this, see \cite{eil}.

\par We will show that $\mathcal F$ is both $\mathcal D_r(2,t)$-free and $\mathcal D_r(t-2,t)$-free. This will follow from the fact that any copy of a $\mathcal D_r(2,t)$-daisy or a $\mathcal D_r(t-2,t)$-daisy in $(\mathbb{F}_q^r \setminus \{0\})^{(r)}$ must contain a linearly dependent set. 
\par Suppose for a contradiction that we have a copy of a $\mathcal D_r(2,t)$-daisy in $(\mathbb{F}_q^r \setminus \{0\})^{(r)}$, whose $r$-sets are all bases of $\mathbb{F}_q^r$. Let $S = \{v_1,v_2,\dots,v_{r-2}\}$ denote the stem of this daisy; then $S$ is a linearly independent set. Let $W$ denote the subspace spanned by $S$, so $W$ is an $(r-2)$-dimensional subspace of $\mathbb{F}_q^r$, and $\dim(\mathbb{F}_q^r/W)=2$. Let $u_1,u_2,\dots,u_{q+2}$ denote the other $t=q+2$ vertices of the daisy.
\par Since ${u_i, u_j,v_1,\dots,v_{r-2}}$ form a basis for all $1\leq i < j \leq q+2$, the vectors $u_1+W,u_2+W,\ldots,u_{q+2}+W$ can be viewed as $q+2$ vectors in $\mathbb{F}_q^2$ of which any two are linearly independent, i.e., no two are on the same line through the origin. But this is impossible, as the $q+1$ lines through the origin in $\mathbb{F}_q^2$ cover $\mathbb{F}_q^2$.

\par Similarly, suppose for a contradiction that we have a copy of a $\mathcal D_r(t-2,t)$-daisy in $\mathcal F$. Let $A=\{y_1,\dots,y_{r-q}\}$ be its stem and $B=\{x_1,\dots,x_{q+2}\}$ be its other vertices. Let $W'$ denote the subspace spanned by $A$, so $W'$ is an $(r-q)$-dimensional subspace of $\mathbb{F}_q^r$, and $\dim(\mathbb{F}_q^r/W') = q$. Then the vectors $x_1+W',x_2+W',\ldots,x_{q+2}+W'$ can be viewed as $q+2$ vectors in $\mathbb{F}_q^q$ of which any $q$ are linearly independent. But this is also easily seen to be impossible. Indeed, relabelling for clarity, suppose for a contradiction that $z_1,\ldots,z_{q+2}$ are $q+2$ vectors in $\mathbb{F}_q^q$ of which any $q$ are linearly independent. We may assume, by change of basis if necessary, that $z_1,\ldots,z_q$ are the first $q$ unit vectors (in order). Then $z_{q+1}$ cannot have any entry equal to zero (if the $i$th entry of $z_{q+1}$ were equal to zero, then $\{z_1,\ldots,z_q,z_{q+1}\}\setminus \{z_i\}$ would be linearly dependent, a contradiction). Hence, we may assume, by rescaling the first $q$ vectors if necessary, that $z_{q+1}$ is the all-ones vector. But then, by the same argument, $z_{q+2}$ cannot have any entry equal to zero, so by the pigeonhole principle two of its entries (say the $i$th and the $j$th) must be equal. By rescaling $z_{q+2}$, we may assume that these two entries are equal to 1. But then $z_{q+2}-z_{q+1} \in \text{Span}(\{z_1,\ldots,z_q\}\setminus\{z_i,z_j\})$ so $\{z_1,\ldots,z_{q+2}\}\setminus\{z_i,z_j\}$ is linearly dependent, a contradiction.

\par For $n=m(q^r-1)$, where $m$ is any positive integer, the same blow-up construction and estimates as in the proof of Theorem 5 in \cite{eil} now give a family $\mathcal F$ that is both $\mathcal D_r(2,t)$-free and $\mathcal D_r(t-2,t)$-free and has density $|\mathcal F|/\binom{n}{r}>\left(1-\binom{r}{2}/(q^r-1)\right)\prod_{k=1}^r(1-q^{-k})$. Indeed, partition $[n]$ into $q^r-1$ classes of equal size, $C_1,\ldots,C_{q^r-1}$ say, and consider the family $\mathcal G$ consisting of all $r$-element subsets of $[n]$ of the form $\{x_1,\ldots,x_r\}$, where $x_i \in C_{j_i}$ for all $i \in [r]$ and $\{j_1,j_2,\ldots,j_r\}$ 
is some element of $\mathcal F$ (the above $\mathcal D_r(2,t)$-free and $\mathcal D_r(t-2,t)$-free family on $[q^r-1]$). Clearly $\mathcal G$ is both $\mathcal D_r(2,t)$-free and $\mathcal D_r(t-2,t)$-free, since $\mathcal F$ is. Since the probability that a uniformly random $r$-element subset of $[n]$ has at least two points in some $C_i$ is at most $\frac{{r \choose 2}}{q^r-1}$, the size of $\mathcal G$ is at least
$$\left(1-\frac{\binom{r}{2}}{q^r-1}\right)|\mathcal F|>\left(1-\frac{\binom{r}{2}}{q^r-1}\right)\prod_{k=1}^r(1-q^{-k}){n \choose r},$$
using (\ref{eq:bases-estimate}).
\par Since $m$ can be taken arbitrarily large, it follows from the observation at the beginning of the proof that $\ex(n,r,t)/\binom{n}{r}>\left(1-\binom{r}{2}/(q^r-1)\right)\prod_{k=1}^r(1-q^{-r})$ for all $n\geq r$.
\par An easy averaging argument shows that $(r+1)\ex(n+1,r+1,t)\leq (n+1)\ex(n,r,t)$, which implies that $\ex(n+1, r+1,t)/\binom{n+1}{r+1}\leq \ex(n,r,t)/\binom{n}{r}$. Therefore we must have
$$\ex(n,r,t)/\binom{n}{r}\geq \ex(n+s,r+s,t)/\binom{n+s}{r+s}>\left(1-\frac{\binom{r+s}{2}}{q^{r+s}-1}\right)\prod_{k=1}^{r+s}(1-q^{-k})$$
for all $s\in \mathbb{N}$. Letting $s\rightarrow\infty$, we obtain $ex(n,r,t)/\binom{n}{r}\geq\prod_{k=1}^{\infty}(1-q^{-k})$, as required.
\end{proof}

\begin{theorem}
\label{thm:q7} The vertex-Tur\'an density of $Q_7$ is greater than $7/8$. Equivalently, $\gamma_7<1/8$.
\end{theorem}

\begin{proof} We begin by noting the well-known fact that $g(n,7)/{2^n}$, the minimal density of a subset of $\{0,1\}^n$ that intersects the vertex-set of any $Q_7$, is increasing as a function of $n$. Indeed, this is the standard averaging argument: let $\mathcal F \subset \{0,1\}^{n+1}$ be a family of size $g(n+1,7)$ that intersects the vertex-set of every $Q_7$. Consider $\mathcal F'=\{v\in\mathcal F: v_{n+1}=0\}$ and $\mathcal F''=\{v\in\mathcal F:v_{n+1}=1\}$. Then both $\mathcal F'$ and $\mathcal F''$, viewed as subsets of $\{0,1\}^n$, intersect the vertex-set of any $Q_7$, and so their sizes are at least $g(n,7)$. Hence $g(n+1,7)=|\mathcal F'|+|\mathcal F''|\geq 2g(n,7)$, implying the monotonicity property claimed. To prove the theorem, it therefore suffices to show that $g(n,7)/{2^n}$ is less than some constant $c<1/8$ for arbitrarily large values of $n$.
\par Let $n=4l-1$ for some positive integer $l$. We construct a small family $\mathcal G_n \subset \{0,1\}^n$ intersecting the vertex-set of any $Q_7$, as follows: at each level $r$ of $Q_n$ that is not congruent to 3 mod.\ 4, we take nothing, and at each level $r$ that is congruent to 3 mod.\ 4 we take the complement of the family $\mathcal F_n^r$ constructed in Proposition \ref{dual}, setting $q=4$. In other words,
$$\mathcal G_n=\bigcup_{0 \leq r' \leq l-1}\{A\in\{0,1\}^n: |A|=3+4r'\text{ and } A\notin\mathcal F_n^{3+4r'}\}.$$
\par We claim that $\mathcal G_n$ intersects the vertex-set of any $Q_7$. Indeed, given a copy of $Q_7$ and looking at its $2^{\text{nd}}$, $3^{\text{rd}}$, $4^{\text{th}}$ and $5^{\text{th}}$ layers, we notice that one of these four layers must be in a layer of $Q_n$ that is congruent to $3\mod4$. Call that layer $r$. By construction, the intersection of $\mathcal G_n$ with the $r$th layer of $Q_n$ is equal to the complement of $\mathcal F_n^r$, and therefore it is a family that intersects any $\mathcal D_r(2,6)$-daisy and any $\mathcal D_r(4,6)$-daisy.
\par Now, notice that layer 2 of a $Q_7$ sitting in layer $r$ of $Q_n$ is a $\mathcal D_r(2,7)$-daisy, which trivially contains a $D_r(2, 6)$-daisy, and so it intersects the complement of $\mathcal F_n^r$. Layer 3 of a $Q_7$ sitting in layer $r$ of $Q_n$ is a $\mathcal 
D_r(3,7)$-daisy, which contains a $\mathcal D_r(2,6)$-daisy (by putting a vertex of its petal-set into its stem), so it also intersects the complement of $\mathcal F_n^r$. Layer 4 of a $Q_7$ sitting in layer $r$ of $Q_n$ is a $\mathcal D_r(4,7)$-daisy, which trivially contains a $\mathcal D_r(4,6)$-daisy, and so it also intersects the complement of $\mathcal F_n^r$. And finally, layer 5 of a $Q_7$ that sits in layer $r$ of $Q_n$ is a $\mathcal D_r(5,7)$-daisy, which by the same argument as in the layer 3 case, contains a $\mathcal D_r(4,6)$-daisy, so it also intersects the complement of $\mathcal F_n^r$. Hence, $\mathcal G_n$ does indeed intersect the vertex-set of any $Q_7$, as claimed.
\par Let us now compute the density of $\mathcal G_n$. We have
$$\frac{|\mathcal G_n|}{2^n}=\frac{\binom{n}{3}-|\mathcal F_n^3|+\binom{n}{7}-|\mathcal F_n^7|+\dots+\binom{n}{4l-1}-|\mathcal F_n^{4l-1}|}{2^n}\leq\frac{\sum_{r'=0}^{l-1}\binom{n}{4r'+3}(1-\prod_{k=1}^{\infty}(1-4^{-k}))}{2^n}.$$
Noting that the (far) right-hand side is $\left(\frac{1}{4}+o(1)\right)\left(1-\prod_{k=1}^{\infty}(1-4^{-k})\right)$, and that $\prod_{k=1}^{\infty}(1-4^{-k})>0.6$, we obtain $|\mathcal G_n|/2^n\leq 1/10+o(1)$. Taking $l$ (and thus $n$) going to infinity, we obtain $\gamma_7\leq 1/10 <1/8$, as required.
\end{proof}

\section{The Tur\'an density of $Q_6$}
\par We now turn to $Q_6$, where the construction is rather more complicated. The key idea will be as follows: rather than considering each layer separately, as above (in the sense that what we take in a given layer does not relate in any way with what we take in a nearby layer), we take into account the interplay between consecutive
layers. In a certain sense (see the proof of Proposition \ref{2levels}), this will allow us to artificially `add' one more vector to our family. For example, in the proof of Proposition \ref{dual} we had a family of $q+2$ vectors of a certain form living in a 2-dimensional vector space over $\mathbb{F}_q$ and we used the fact that some two must be linearly dependent, but by this device of looking at consecutive layers we would get $q+3$ vectors instead. The precise details are in Proposition \ref{2levels} below, with the corresponding facts about linear dependence being in Lemma \ref{vectors3} and Lemma \ref{vectors4}. The actual way we link the two levels is an idea from \cite{gm}.
\par For $n\geq r\geq 1$ and $2\leq i\leq5$, we say that a family $\mathcal F \subset [n]^{(r-1)}\cup [n]^{(r)}$ is a \textit{copy of the layers $i$ and $i-1$ of a $Q_6$ in $Q_n$} if there exists sets $X$ of size $6$ and a set $Y$ of size $r-i$ with $X \cap Y = \emptyset$, such that $\mathcal F=\{Y\cup A:\ A \in X^{(i-1)}\}\cup \{Y\cup A:\ A \in X^{(i)}\}$. A {\em copy of two consecutive nontrivial layers of a $Q_6$ in $Q_n$} refers to one of these families for $2 \leq i \leq 5$. If $\mathcal{F} \subset [n]^{(r-1)} \cup [n]^{(r)}$, the {\em density-sum} of $\mathcal{F}$ is defined to be $|\mathcal{F} \cap [n]^{(r-1)}|/{n \choose r-1}+|\mathcal{F} \cap [n]^{(r)}|/{n \choose r}$, i.e., it is the sum of the densities of $\mathcal{F}$ inside its two respective layers.

\begin{lemma}Let $v_1,v_2,\dots,v_7$ be seven vectors in $\mathbb{F}_5^3$. Then there exist three of them that do not form a basis.
\begin{proof} Suppose for a contradiction that there exist seven vectors $v_1,v_2,\ldots,v_7 \in \mathbb{F}_5^3$, any three of which form a basis. By a change of basis, we may assume that $v_1=(1,0,0)$, $v_2=(0,1,0)$ and $v_3=(0,0,1)$. We now observe that the other vectors cannot have any entry equal to 0. Indeed, if for example $v_4=(a,b,0)$, then $v_1,v_2,v_4$ span a 2-dimensional space, thus they cannot form a basis. Therefore, rescaling if necessary, we may assume that one other vector is $(1,1,1)$, say $v_4$, and that the first entry in all the other three vectors (which are $v_5,v_6$ and $v_7$), is equal to 1.
\par Looking now at the second entries of the vectors $v_5,v_6,v_7$ we notice that no two of these entries can be equal (otherwise the corresponding vectors and $v_3$ would span a 2-dimensional space), and moreover that none of them can be equal to 1 (otherwise the corresponding vector, $v_3$ and $v_4$ would span a 2-dimensional space); and we have also already observed that they cannot be equal to zero. Therefore, these second entries of $v_5,v_6$ and $v_7$ must be, in some order, 2, 3 and 4. Without loss of generality let $v_5=(1,2,x)$, $v_6=(1,3,y)$ and $v_7=(1,4,z)$, where $x,y,z\in\mathbb F_5$. The above observation also applies to the third entries of these vectors, and therefore $x,y,z$ must also be $2,3,4$ in some order. Moreover, $x\neq2$, $y\neq3$ and $z\neq 4$, otherwise the corresponding vector lies in the span of $v_1$ and $v_4$. Therefore we only have two possibilities: either $x=3,y=4,z=2$, or else $x=4, y=2,z=3$.
\par In the first case, we have $v_6=(1,3,4)=2(1,4,2)-(1,0,0)=2v_7-v_1$, so $v_1,v_6,v_7$ do not form a basis. In the second case, we have $v_7=(1,4,3)=2(1,2,4)-(1,0,0)=2v_5-v_1$, so $v_1,v_5,v_7$ do not form a basis.
\par We have reached a contradiction in all cases, proving the lemma.
\end{proof}
\label{vectors3}
\end{lemma}
\begin{lemma}Let $v_1,v_2,\dots,v_7$ be seven vectors in $\mathbb{F}_5^4$. Then there exist four of them that do not form a basis.
\begin{proof} Suppose for a contradiction that there exist seven vectors $v_1,v_2,\ldots,v_7 \in \mathbb{F}_5^4$, any four of which form a basis. By a change of basis, we may assume that four of them are $v_1=(1,0,0,0)$, $v_2=(0,1,0,0)$, $v_3=(0,0,1,0)$ and $v_4=(0,0,0,1)$. As before, since the remaining vectors cannot have any entry equal to 0, we may also assume that $v_5=(1,1,1,1)$ and that the first entry of both $v_6$ and $v_7$ is 1. Moreover, by (essentially) the same argument as in Lemma \ref{vectors3}, under the above assumptions, $v_6$ and $v_7$ cannot have two entries equal, and neither can they have any entry other than the first being equal to 1, thus their second, third and fourth entries are $2,3,4$ in some order. Without loss of generality we may assume that $v_6=(1,2,3,4)$ and that $v_7=(1,a,b,c)$, where $\{a,b,c\}=\{2,3,4\}$. Furthermore, $a\neq2$, otherwise $v_6,v_7,v_3,v_4$ would span a space of dimension 3. Similarly, $b\neq3$ and $c\neq4$. This leaves us with only two options, $v_7=(1,3,4,2)$ or $v_7=(1,4,2,3)$.
\par In the first case, we have $2v_7=(2,1,3,4)=v_6+v_1-v_2$, thus $v_1,v_2,v_6,v_7$ do not form a basis. In the second case, we have $2v_6=(2,4,1,3)=v_7+v_1-v_3$, thus $v_1,v_3,v_6,v_7$ do not form a basis.
\par We have reached a contradiction in all cases, proving the lemma.
\end{proof}
\label{vectors4}
\end{lemma}
\begin{proposition}
\label{2levels} 
Let $n\geq r\geq 1$. Then there exist two families $\mathcal F_1\subset [n]^{(r)}$ and $\mathcal F_2\subset [n]^{(r-1)}$ such that $\mathcal F_1\cup\mathcal F_2$ does not contain a copy of the layers $i$ and $i-1$ of a $Q_6$ for any $2\leq i\leq 5$, and $|\mathcal F_1|/\binom{n}{r}+|\mathcal F_2|/\binom{n}{r-1}>1.52$.
\end{proposition}
\begin{proof} Let $l(n,6,r)$ denote the maximum density-sum of a family $\mathcal F \subset [n]^{(r-1)} \cup [n]^{(r)}$, that does not contain a copy of layers $i$ and $i-1$ of a $Q_6$, for any $2\leq i\leq 5$. In other words, $$l(n,6,r)=\max\{{|\mathcal F_1|/\binom{n}{r}+|\mathcal F_2|/\binom{n}{r-1}}: \mathcal F_1\subset [n]^{(r)}, \mathcal F_2\subset [n]^{(r-1)} \text{ and $\mathcal F_1\cup\mathcal F_2$ does not contain}$$ a copy of the layers $i$ and $i-1$ of a $Q_6$ for any $2\leq i\leq 5$\}.
\par We first show that $l(n,6,r)$ is decreasing as a function of $n$. Let $\mathcal F=\mathcal F_1\cup\mathcal F_2$ be a family not containing the forbidden configuration that achieves the maximal density-sum subject to this condition, with $\mathcal{F}_1 \subset [n]^{(r)}$ and $\mathcal{F}_2 \subset [n]^{(r-1)}$, i.e., $|\mathcal F_1|/\binom{n}{r}+|\mathcal F_2|/\binom{n}{r-1}=l(n,6,r)$. Note that for every $i\in[n]$, the family $\mathcal F_i=\{A\in\mathcal F:i\notin A\}$ also does not contain the forbidden configuration, and has ground-set of size $n-1$, so its density-sum is at most $l(n-1,6,r)$. Summing the density-sums over all $i$ we get $l(n,6,r)\leq l(n-1,6,r)$, as required.
\par The second observation we make is that $l(n,6,r)\geq l(n+1,6,r+1)$. Indeed, given a family $\mathcal F \subset [n+1]^{(r+1)} \cup [n+1]^{(r)}$ not containing the forbidden configuration, for each $i \in [n+1]$ the family $\mathcal F^i=\{A\setminus\{i\}: i\in A, A\in\mathcal F\}$ satisfies the same condition with $(n,r)$ in place of $(n+1,r+1)$, so the density-sum of $\mathcal{F}^i$ is at most $l(n,6,r)$. Taking $\mathcal{F}$ with maximal density-sum and then summing over all $i$, we obtain $l(n+1,6,r+1)\leq l(n,6,r)$.
\par We will now show that for fixed $r$ and arbitrarily large $n$ there exist a family that is free of any two consecutive nontrivial layers of a $Q_6$ and has density-sum at least $(1-\binom{r}{2}/(5^r-1))\prod_{k=1}^{r}(1-5^{-k})+(1-\binom{r-1}{2}/(5^r-1))\prod_{k=1}^{r-1}(1-5^{-k})$. This, together with the two monotonicity observations above, will yield $l(n,6,r)\geq2\prod_{k=1}^{\infty}(1-5^{-k})>1.52$, as in the proof of Theorem \ref{thm:q7}.
\par To start with, let $n=5^r-1$ and identify $[n]$ with $\mathbb F_5^r\setminus\{0\}$. Fix any vector $w \in \mathbb F_5^r\setminus\{0\}$ and let $\mathcal F_1 \subset (\mathbb{F}_5^r\setminus\{0\})^{(r)}$ be the set of all bases, and $\mathcal F_2 \subset (\mathbb{F}_5^r \setminus \{0\})^{(r-1)}$ be the set of all $(r-1)$-sets that together with $w$ form a basis. The density of $\mathcal F_1$ is greater than $\prod_{k=1}^r(1-5^{-k})$, and the density of $\mathcal F_2$ is greater than $\prod_{k=1}^{r-1}(1-5^{-k})$, thus the density-sum of $\mathcal F=\mathcal F_1\cup\mathcal F_2$ is greater than $\prod_{k=1}^{r}(1-5^{-k})+\prod_{k=1}^{r-1}(1-5^{-k})$.
\par Suppose for a contradiction that $\mathcal F$ contains a copy of the layers 1 and 2 of a $Q_6$. This means that there exist 6 vectors $v_1, v_2,\dots,v_6$ and $r-2$ other vectors, say $w_1,\dots,w_{r-2}$, such that $\{x,y,w_1,\dots,w_{r-2}\}$ is a basis for any two distinct $x,y\in\{w,v_1,\dots,v_6\}$. Let $W$ be the span of $\{w_1,\dots,w_{r-2}\}$, which by assumption must be an $(r-2)$-dimensional vector space over $\mathbb F_5$. Now, the cosets $w+W, v_1+W,\dots v_6+W$ can be viewed as 7 vectors in $\mathbb F_5^2$ of which any 2 form a basis. Without loss of generality, by changing basis if necessary we may assume that two of them are $(1,0)$ and $(0,1)$. This means that no entry in the remaining 5 vectors can be 0, so by rescaling we may assume that another of the five vectors is $(1,1)$ and all the remaining four have first entry equal to 1. In view of this, we cannot have the second entry of any of the remaining four vectors being a 0 or a 1, thus it can only be 2, 3 or 4. But this is clearly impossible since we have four distinct vectors, a contradiction.
\par Now suppose that $\mathcal F$ contains a copy of the layers 4 and 5 of a $Q_6$. This means that there exists 6 vectors $v_1,v_2,\dots,v_6$ and other $r-5$ vectors, say $w_1,\dots w_{r-5}$, such that for any pairwise distinct $x,y,z,t,u\in\{w,v_1,\dots v_6\}$, the set $\{x,y,z,t,u,w_1,\dots w_{r-5}\}$ forms a basis. As before, this implies the existence of 7 vectors in $\mathbb F_5^5$ with the property that any 5 of them form a basis. Again, we may assume that the first five of them are $(1,0,0,0,0)$, $(0,1,0,0,0)$, $(0,0,1,0,0)$, $(0,0,0,1,0)$ and $(0,0,0,0,1)$. Since the other two remaining vectors can have no zero component, by rescaling we may also assume that one of the other vectors is $(1,1,1,1,1)$ and the other one has the first component 1. We now observe that no two components in the seventh vector can be equal, thus they can now only be 2, 3, or 4. However, we have 4 slots to fill in the seventh vector, a contradiction.
\par Finally, if $\mathcal F$ contains a copy of layers 2 and 3 of a $Q_6$ then there would exist seven vectors in $\mathbb F_5^3$ of which any three form a basis, which is impossible by Lemma \ref{vectors3}. Similarly, if $\mathcal F$ contains a copy of layers 3 and 4 of a $Q_6$ then there would exist seven vectors in $\mathbb F_5^4$ of which any four form a basis, which is impossible by Lemma \ref{vectors4}.
\par Now we perform the same blow-up procedure as before. Let $n=m(5^r-1)$ for some arbitrary positive integer $m$. We partition $[n]$ into $5^r-1$ classes of equal size, $C_1,\ldots,C_{5^r-1}$ say, and consider the family $\mathcal G_1$ consisting of all $r$-element subsets of $[n]$ of the form $\{x_1,\dots,x_r\}$, where $x_i \in C_{j_i}$ for all $i \in [r]$, for $\{j_1,j_2,\dots,j_r\}$ some element of $\mathcal F_1$, and the family $\mathcal G_2$ consisting of all the $(r-1)$-element subsets of $[n]$ of the form $\{x_1,\dots,x_{r-1}\}$, where $x_i\in C_{j_i}$ for all $i\in[r-1]$, for $\{j_1,\dots,j_{r-1}\}$ some element of $\mathcal F_2$. Since the family $\mathcal F=\mathcal F_1\cup \mathcal F_2$ is free of any two consecutive nontrivial layers of $Q_6$, so is the `blow-up' family $\mathcal G=\mathcal G_1\cup\mathcal G_2$. Moreover, observing that the probability that a uniformly random $t$-element subset of $[n]$ has at least two points in some $C_i$ is at most $\frac{{t \choose 2}}{5^r-1}$, we see that the density of $\mathcal G_1$ is at least $\left(1-\binom{r}{2}/(5^r-1)\right)|\mathcal F_1|>\left(1-\binom{r}{2}/(5^r-1)\right)\prod_{k=1}^r(1-5^{-k})$, and that the density of $\mathcal G_2$ is at least  $\left(1-\binom{r-1}{2}/(5^r-1)\right)|\mathcal F_2|>\left(1-\binom{r-1}{2}/(5^r-1)\right)\prod_{k=1}^r(1-5^{-k})$. Hence, the density-sum of $\mathcal G_1\cup\mathcal G_2$ is at least $\left(1-\binom{r}{2}/(5^r-1)\right)\prod_{k=1}^r(1-5^{-k})+\left(1-\binom{r-1}{2}/(5^r-1\right)\prod_{k=1}^{r-1}(1-5^{-k})$, which finishes the proof, as $m$ can be taken to be arbitrarily large.
\end{proof}
\begin{theorem} The vertex-Tur\'an density of $Q_6$ is greater than $6/7$. Equivalently, $\gamma_6<1/7$.
\end{theorem}
\begin{proof}
We begin by noting that $g(n,6)/{2^n}$, the minimal size of a subset of $\{0,1\}^n$ that intersects any vertex-set of a $Q_6$, is increasing. Indeed, the proof of this is the same as for $Q_7$ (at the start of the proof of Theorem \ref{thm:q7}.)
\par Let $n=8s-1$ for some positive integer $s$. We construct a small family $\mathcal{G}_n \subset \{0,1\}^n$ intersecting the vertex-set of any $Q_6$, as follows: on levels $r$ that are congruent to 0 or 1 mod 4 we take nothing, while on levels $r$ and $r-1$ for $r$ congruent to 3 mod 4, we take the complements of two families $\mathcal F_1^r\subset[n]^{(r)}$ and $\mathcal F_2^{r-1}\subset[n]^{(r-1)}$ constructed in Proposition \ref{2levels}.
\par First, if a copy of $Q_6$ starts in layer $i$ of $Q_n$, then one of the numbers $i+1, i+2, i+3, i+4$, say $i'$, is congruent to $2\mod4$, which implies that its (nontrivial) layers that live in layers $i'$ and $i'+1$ of $Q_n$ intersect $([n]^{(i'+1)}\setminus\mathcal F_1^{i'+1})\cup ([n]^{(i')}\setminus\mathcal F_2^{i'})$, and consequently $\mathcal G_n$. Therefore $\mathcal G_n$ intersects every copy of $Q_6$.
\par The density of $\mathcal G_n$ is: $$\frac{|\mathcal G_n|}{2^n}=\frac{\sum_{i=0}^{2s-1}\binom{n}{4i+2}+\binom{n}{4i+3}-|\mathcal F_1^{4i+3}|-|\mathcal F_2^{4i+2}|}{2^n}=\frac{\sum_{i=0}^{2s-1}\binom{n}{4i+2}+\binom{n}{4i+3}}{2^n}-\frac{\sum_{i=0}^{2s-1}|\mathcal F_1^{4i+3}|+|\mathcal F_2^{4i+2}|}{2^n},$$ which is equal to $\frac{1}{2}+o(1)-\frac{\sum_{i=0}^{2s-1}|\mathcal F_1^{4i+3}|+|\mathcal F_2^{4i+2}|}{2^n}=\frac{1}{2}+o(1)-\frac{\sum_{i=0}^{s-1}|\mathcal F_1^{4i+3}|+|\mathcal F_2^{4i+2}|}{2^n}-\frac{\sum_{i=s}^{2s-1}|\mathcal F_1^{4i+3}|+|\mathcal F_2^{4i+2}|}{2^n}$.
\par However, since for $r\leq 4s-1<n/2$ we have $\binom{n}{r}>\binom{n}{r-1}$, we obtain $\frac{|\mathcal F_1^r|+|\mathcal F_2^{r-1}|}{\binom{n}{r-1}}>\frac{|\mathcal F_1^r|}{\binom{n}{r}}+\frac{|\mathcal F_2^{r-1}|}{\binom{n}{r-1}}>1.52$. Similarly, for  $r\geq 4s+1>n/2+1$, we obtain $\frac{|\mathcal F_1^r|+|\mathcal F_2^{r-1}|}{\binom{n}{r}}>\frac{|\mathcal F_1^r|}{\binom{n}{r}}+\frac{|\mathcal F_2^{r-1}|}{\binom{n}{r-1}}>1.52$. Putting everything together, we obtain: $$\frac{|\mathcal G_n|}{2^n}\leq\frac{1}{2}+o(1)-1.52\frac{\sum_{i=0}^{s-1}\binom{n}{4i+2}}{2^n}-1.52\frac{\sum_{i=s}^{2s-1}\binom{n}{4i+3}}{2^n}=\frac{1}{2}+o(1)-1.52(1/8+o(1))-1.52(1/8+o(1)).$$
\par Finally, letting $s\rightarrow\infty$ we obtain $\gamma_6\leq 1/2-1.52/4=0.12<1/7$, as required.
\end{proof}
\section{Concluding remarks}
\par The remaining cases are $d=3,4,5$. Perhaps the most interesting of these is $d=3$, as that is the
one where the methods of this paper seem to have the least chance of succeeding. 
\par One could try to attack these questions by trying to generalize the linking of two consecutive levels to more consecutive levels. For example, in the above we have considered those $r$-sets that form a basis of a certain vector space, and then we also had those $(r-1)$-sets which, together with a fixed vector $v$, form a basis -- and this was what allowed us to `gain a vector for free'. What if we add in those $(r-2)$-sets which, together with $v$ and one other fixed vector $w$ (with $v$ and $w$ linearly independent), form a basis? Sadly, this does not allow us to gain two vectors, since we would \textit{not} know that original vectors with $w$ alone form bases. We do not see how to get around this obstacle.
\par It would also be of great interest to determine the growth speed of $\gamma_d$. It is known from
\cite{eil} that the decay is exponential in $d$, and it is trivial that $\gamma_d \geq 1/2^d$, but
could it be true that this trivial lower bound gives the right speed? In other words, could it be
that $\gamma_d = (1/2 +o(1))^d$?

\Addresses

\begin{thebibliography}{99}
\bibitem{blm} B. Bollob\'as, I. Leader and C. Malvenuto. Daisies and other Tur\'an problems. \textit{ Combinatorics, Probability and Computing} 20 (2011), 743--747.
\bibitem{eil} D. Ellis, M.-R. Ivan and I. Leader. Tur\'an densities for daisies and hypercubes. \textit{Bulletin of the London Mathematical Society} (2024), \url{https://doi.org/10.1112/blms.13171}.
\bibitem{gm} A. Grebennikov and J. P. Marciano. $C_{10}$ has positive Tur\'an density. \textit{Journal of Graph Theory} 109 (2025), 31 -- 34.
\bibitem{je} K.A. Johnson and R. Entringer. Largest induced subgraphs of the $n$-cube that
contain no 4-cycles. {\em Journal of Combinatorial Theory, Series B} 46 (1989), 346--
355.
\bibitem{k} E. A. Kostochka. Piercing the edges of the $n$-dimensional unit cube. Diskret.
Analiz Vyp. 28 Metody Diskretnogo Analiza v Teorii Grafov i Logiceskih
Funkcii (1976), 55--64. [In Russian.]
\end{thebibliography}
\end{document}